\documentclass[11pt,english]{amsart}
\usepackage[T1]{fontenc}
\usepackage[latin9]{inputenc}
\usepackage{prettyref}
\usepackage{mathtools}
\usepackage{enumitem}
\usepackage{amstext}
\usepackage{amsthm}
\usepackage{amssymb}
\usepackage{stmaryrd}

\makeatletter
\numberwithin{equation}{section}
\numberwithin{figure}{section}
\theoremstyle{plain}
\newtheorem{thm}{\protect\theoremname}[section]
\theoremstyle{plain}
\newtheorem{cor}[thm]{\protect\corollaryname}
\theoremstyle{plain}
\newtheorem{prop}[thm]{\protect\propositionname}

\makeatletter
\@addtoreset{thm}{section}
\makeatother
\usepackage{prettyref}
\usepackage{refstyle}
\usepackage{hyperref}
\usepackage{url}
\usepackage{mathtools}
\newref{claim}{refcmd={claim \ref{#1}}}

\makeatother

\usepackage{babel}
\providecommand{\corollaryname}{Corollary}
\providecommand{\propositionname}{Proposition}
\providecommand{\theoremname}{Theorem}

\begin{document}
\title{The eta invariant of a circle bundle on a Fano manifold}
\author{Nikhil Savale}
\address{School of Mathematics, Trinity College Dublin, Dublin 2, Ireland}
\email{savale@maths.tcd.ie}
\begin{abstract}
We consider the spin-c Dirac operator on the unit circle bundle of
a positive line bundle over a Fano manifold of even complex dimension.
We compute the corresponding eta invariant in terms of Zhang's value
of its adiabatic limit \cite{Zhang}. This extends the earlier computation
of the author \cite{Savale-Asmptotics} from small to arbitrary values
of the adiabatic parameter. 
\end{abstract}

\maketitle

\section{\label{sec:Introduction} Introduction}

The eta invariant of Atiyah-Patodi-Singer \cite{APSI} was introdiced
as a correction term to an index theorem for manifolds with boundary.
For a first order, elliptic and self-adjoint operator $A$ on a compact
manifold, the eta invariant $\eta(A)$ is formally its signature.
That is, the difference between the number of positive and the number
of negative eigenvalues of $A$. In reality, due to $A$ having infinitely
many eigenvalues, this needs to be defined via regularization (see
\prettyref{subsec:Dirac-operators-and}).

An important feature of the invariant $\eta(A)$, just like the signature
of a matrix, is that it is in general not a continuous function of
the operator $A$. This makes it difficult to compute the eta invariant
explicitly. In this article we present a calculation of the eta invariant
for circle bundles over a Fano manifold.

Let us state the result precisely. Let $X^{n}$ be a compact, complex
manifold of complex dimension $n$. Consider $\left(\mathcal{L},h^{\mathcal{L}}\right)\rightarrow X$
a positive holomorphic, Hermitian line bundle over it. Denote by $\nabla^{\mathcal{L}}$
the corresponding Chern connection and $\omega\coloneqq iR^{\mathcal{L}}\in\Omega^{1,1}\left(X\right)$
its curvature form. The positivity of the bundle defines via $h^{T^{1,0}X}\left(v,w\right)=R^{\mathcal{L}}\left(v,\bar{w}\right)$,
for $v,w\in T^{1,0}X$, a Kahler metric on the complex tangent space.
The corresponding Ricci curvature form is denoted by $\textrm{Ric}_{\omega}$;
this can be defined in local holomorphic coordinates via the expression
$\textrm{Ric}_{\omega}\coloneqq i\partial\bar{\partial}\ln\det\left(g_{j\bar{k}}\right)$,
where $\omega=ig_{j\bar{k}}dz_{j}\wedge d\bar{z}_{k}$. We shall assume
that the Ricci form is positive $\textrm{Ric}_{\omega}>\kappa\omega$
for some $\kappa>0$ (in fact $\kappa\geq0$ is sufficient). Since
$\textrm{Ric}_{\omega}$ is a representative of the Chern curvature
form of the anticanonical bundle $K_{X}^{*}$, this particularly means
that the the anticanonical bundle is ample, i.e. the manifold is Fano. 

Note that such metrics exist in abundance on Fano manifolds: being
projective Fano manifolds admit integral $\left(1,1\right)$ Kahler
forms $\omega'$. And hence admit cohomologous Kahler forms $\left[\omega\right]=\left[\omega'\right]$
with positive Ricci curvature by the Calabi-Yau theorem.

Now let $Y=S^{1}\mathcal{L}\xrightarrow{\pi}X$ denote the bundle
of unit elements of $\mathcal{L}$. The tangent bundle of total space
splits $TY=TS^{1}\oplus T^{H}Y$, via the connection $\nabla^{\mathcal{L}}$,
into the bundles of vertical $TS^{1}$ and horizontal tangent vectors
$T^{H}Y$. Below it is also useful to define the connection form $a\in\Omega^{1}\left(Y\right)$
via $a\left(HX\right)=0$ and $a\left(e\right)=1$, with $e\in TS^{1}$
being the generator of the natural circle action on the fibres. This
now allows us to equip with $Y$ with the family of adiabatic metrics
defined via 
\[
g_{\varepsilon}^{TY}\coloneqq g^{TS^{1}}\oplus\varepsilon^{-1}\pi^{*}g^{TX},\quad\forall\varepsilon>0,
\]
with the horizontal metric $\pi^{*}g^{TX}$ being pulled back from
the base.  

A spin structure on $X$ corresponds to a square root of its canonical
line bundle $K_{X}$, i.e. a holomorphic line bundle $\mathcal{K}$
such that $\mathcal{K}^{\otimes2}=K_{X}$ (see \cite{Hitchin}). The
spin structure on the base then lifts to define a spin structure on
the circle bundle $Y$. The corresponding bundle spinors is then the
pullback to $Y$ of the bundle $S\coloneqq\Lambda T^{0,1*}X\otimes\mathcal{K}$.
This now allows us to define the spin-c Dirac operator 
\begin{equation}
\text{\ensuremath{D_{r,\varepsilon}}}\coloneqq D_{\varepsilon}+rc\left(a\right):C^{\infty}\left(Y;S\right)\rightarrow C^{\infty}\left(Y;S\right),\quad r\in\mathbb{R},\label{eq:Spin Dirac operator}
\end{equation}
on the total space. Above $c\left(a\right)\alpha\coloneqq\left(-1\right)^{\textrm{deg}\alpha}\alpha$,
for $\alpha\in C^{\infty}\left(Y;\Lambda T^{0,1*}X\otimes\mathcal{K}\right)$,
denotes Clifford multiplication by the connection form $a$. The parameter
$r$ is referred to as the semiclassical parameter. The case when
$r=0$ is the spin Dirac operator $D_{\varepsilon}\coloneqq D_{0,\varepsilon}$
and is of special importance. It is related to the spin Dirac operator
on the unit disc bundle $Z\coloneqq\mathbb{D}^{1}\mathcal{L}\coloneqq\left\{ v\in\mathcal{L}|\left|v\right|\leq1\right\} $.
Namely, the unit disc bundle aquires a complex structure and metrix.
The spin structure on $X$ again lifts to a spin structure on the
disc bundle $Z$. And thus gives rise to the spin Dirac operator on
the disc bundle $D^{Z}$.

The eta invariant $\eta_{\varepsilon,r}=\eta\left(D_{r,\varepsilon}\right)$
of the spin-c Dirac operator is formally its signature and defined
via regularization (see \prettyref{subsec:Dirac-operators-and} below).
Our main result is the following computation of the eta invariant
when the real dimension $2n=4m$ is divisible by four.
\begin{thm}
\label{thm:main computation thm} Let $Y$ be the unit circle bundle
of a positive line bundle $\mathcal{L}\rightarrow X$ over a complex
manifold of real dimension $4m$. Suppose that the Ricci curvature
of the Kahler form $\omega=iR^{\mathcal{L}}$ satisfies $\textrm{Ric}_{\omega}\geq\kappa\omega$,
for $\kappa\geq0$. 

The eta invariant of the spin-c Dirac operator $\eta_{\varepsilon,r}=\eta\left(D_{r,\varepsilon}\right)$
\prettyref{eq:Spin Dirac operator} on the total space is then given
by 
\begin{align}
\eta_{\varepsilon,r} & =\frac{1}{2}\int_{X}\hat{A}(X)\,\hat{\eta}_{r}\exp\left\{ rc\right\} \label{eq:formula for eta invariant}\\
 & \qquad-\frac{i}{\left(2\pi i\right)^{m}}\int_{0}^{\varepsilon}d\delta\int_{X}\Omega_{2}\exp\left\{ \Omega_{0}\right\} \exp\left\{ rc\right\} \label{eq:formula for eta invariant line 2}\\
\textrm{where }\qquad\hat{\eta}_{r} & =\begin{cases}
\frac{\exp\left((1-2\{r\})\frac{c}{2}\right)}{\sinh\left(\frac{c}{2}\right)}-\frac{1}{c/2}, & r\notin\mathbb{Z},\\
\left[\frac{\frac{c}{2}-\tanh\left(\frac{c}{2}\right)}{\frac{c}{2}\tanh\left(\frac{c}{2}\right)}\right], & r\in\mathbb{Z},
\end{cases}\label{eq:eta tilde form}\\
\Omega_{0}= & 2\textrm{tr}\left[p\left(R^{TX^{1,0}}+2i\delta\omega\right)\right]+2p\left(2i\delta\omega\right),\\
\textrm{and }\qquad\Omega_{2}= & 2\textrm{tr}\left[ip'\left(R^{TX^{1,0}}+i2\delta\omega\right)\right]+i2p'\left(2i\delta\omega\right),\label{eq: transgression forms}
\end{align}
for each $\varepsilon>0$ and $\left|r\right|\leq\frac{1}{2}\kappa$. 

Here, $p(z)=\frac{1}{2}\log\left(\frac{z/2}{\sinh\left(z/2\right)}\right)$
is the given power series, $\hat{A}(X)$ denotes the $A$-genus of
$X$ while $c=\omega=c_{1}\left(\mathcal{L}\right)$ is the first
Chern form of the positive bundle.
\end{thm}

A particular specialization of the above result is when $r=0$, which
corresponds to the spin Dirac operator. In this case, both terms in
the formula for the eta invariant \prettyref{eq:formula for eta invariant},
\prettyref{eq:formula for eta invariant line 2} vanish. For the first
term, we note that the function involving the hyperbolic tangent in
\prettyref{eq:eta tilde form} is an odd function of $c$. While the
$\hat{A}$ genus is of degree divisible by four. Thus the integrand
$\hat{A}(X)\,\hat{\eta}_{r}$ is a form whose degree is two modulo
four whose integral would vanish over the $4m$ dimensional manifold
$X$. A similar argument using the evenness of the function $p(z)=\frac{1}{2}\log\left(\frac{z/2}{\sinh\left(z/2\right)}\right)$
shows that the second term \prettyref{eq:formula for eta invariant line 2}
vanishes too. This gives the following corollary.
\begin{cor}
\label{cor: main corollary} Let $Y$ be the unit circle bundle of
a positive line bundle $\mathcal{L}\rightarrow X$ over a complex
manifold of real dimension $4m$. Suppose that the Ricci curvature
of the Kahler form $\omega=iR^{\mathcal{L}}$ is semipositive $\textrm{Ric}_{\omega}\geq0$.
The eta invariant of the spin Dirac operator $\eta_{\varepsilon}\coloneqq\eta\left(D_{\varepsilon}\right)=0$
\prettyref{eq:Spin Dirac operator} on the total space is vanishes. 

Consequently, the Atiyah-Patodi-Singer index of the spin Dirac operator
on the unit disc bundle $Z$ is given by 
\begin{equation}
\textrm{ind}\left(D^{Z}\right)=-\frac{1}{2}\sum_{\begin{subarray}{l}
\;k\in\mathbb{Z},\,0\leq p\leq n\\
k+\varepsilon\left(p-\frac{n}{2}\right)=0
\end{subarray}}\textrm{dim}H^{p}\left(X;\mathcal{K}\otimes L^{k}\right).\label{eq:APS index}
\end{equation}
\end{cor}

A main component in the proof of the above is a Kodaira type 'spin
vanishing theorem' for the positive line bundle $\mathcal{L}$ twisted
by the square root bundle $\mathcal{K}$ proved in \prettyref{sec:Spin-vanishing-and}.
It via a corresponding Nakano estimate via the Bochner-Kodaira-Nakano
formula. The proof of this estimate is where the Fano hypothesis on
the manifold $X$ is used. In the final \prettyref{sec:Manifolds-of-general-type}
we shall give examples of manifold of general type where the similar
vanishing theorem does not hold.

The eta invariant has been computed, and its behaviour investigated,
in various cases (see the survey \cite{Goette-2012-etasurvey}). Relevant
to the discussion here is the adiabatic limit of the eta invariant
over fiber bundles whose existence was shown in \cite{Bismut-Cheeger,Dai}.
For general circle bundles the limit has been computed by Zhang in
\cite{Zhang}.

Our result should be contrasted with the one proved by the author
in \cite[Thm 5.7]{Savale-Asmptotics}. Here the above eta invariant
$\eta_{\varepsilon,r}$ was computed for arbitrary values of the semiclassical
parameter $r$ but small values of the adiabatic parameter $\varepsilon$.
The smallness of $\varepsilon$ is quantified in terms of the bottom
of the spectrum of the Kodaira Laplacian acting of tensor powers of
$\mathcal{L}$. The computation used Zhang's formula for the adiabatic
limit \cite{Zhang}, as indeed the first line on the right hand side
of \prettyref{eq:formula for eta invariant} in the adiabatic limit
$\lim_{\varepsilon\rightarrow0}\eta_{\varepsilon,r}$. The computation
in \cite{Savale-Asmptotics} was furthermore used by the author in
\cite{Savale2017-Koszul,Savale-Gutzwiller,Savale2020-hyperbolicity}
to show the sharpness of the asymptotics of the eta invariant in the
semi-classical limit $r\rightarrow\infty$. In contrast, our main
result \prettyref{thm:main computation thm} here computes the same
eta invariant $\eta_{\varepsilon,r}$ for small values of the semiclassical
parameter $r$ but arbitrary values of the adiabatic parameter $\varepsilon$.
The smallness of the semiclassical parameter $r$ is quantified by
the lower bound on the Ricci curvature. 

The article is organized as follows. In \prettyref{sec:Preliminaries}
we begin with some preliminaries on complex geometry \prettyref{subsec:Complex-geometry}
and Dirac operators \prettyref{subsec:Dirac-operators-and}. In \prettyref{thm: Spin-vanishing-and Nakano}
we prove the important spin vanishing theorem and Nakano estimate
on which the proof is based. This is then used in \prettyref{sec:Computation-of-the}
in proving \prettyref{thm:main computation thm} on the computation
of the eta invariant. In the final \prettyref{sec:Manifolds-of-general-type},
we give an example of a manifold of general type where the main theorem
does not hold; thus showing the necessity of the Fano hypothesis.

\section{\label{sec:Preliminaries} Preliminaries}

In this section we review some preliminary notions on complex geometry
and Dirac operators that are used in the article.

\subsection{\label{subsec:Complex-geometry}Complex geometry}

First we begin with some requisite facts from the complex geometry
of positive vector bundles. Standard references for the material below
are \cite{Demailly-published-text,Ma-Marinescu}. Let $X$ be a complex
manifold. Let $\left(\mathcal{L},h^{\mathcal{L}}\right)$ be a holomorphic,
Hermitian line bundle. Its holomorphic structure defines the Dolbeault
operator $\bar{\partial}_{\mathcal{L}}:\Omega^{p,q}\left(X;\mathcal{L}\right)\rightarrow\Omega^{p,q+1}\left(X;\mathcal{L}\right)$
on $\left(p,q\right)$-forms that are valued in sections of the line
bundle. Its Chern connection $\nabla^{\mathcal{L}}:\Omega^{0}\left(X;\mathcal{L}\right)\rightarrow\Omega^{1}\left(X;\mathcal{L}\right)$
is the unique connection on it that is compatible with the Hermitian
metric $h^{\mathcal{L}}$ and whose $\left(0,1\right)$ component
is its holomorphic derivative $\left(\nabla^{\mathcal{L}}\right)^{0,1}=\bar{\partial}_{\mathcal{L}}$.
The curvature of this connection $R^{\mathcal{L}}=\left(\nabla^{\mathcal{L}}\right)^{2}\in\Omega^{1,1}\left(X\right)$
is referred to as the Chern connection. 

The line bundle $\left(\mathcal{L},h^{\mathcal{L}}\right)$ is said
to be \textit{positive} when $iR^{\mathcal{L}}\left(w,\bar{w}\right)>0$
for each non-zero complex tangent vector $w\in T^{1,0}X\setminus0$.
Under the positivity condition, one may define a Hermitian metric
on the complex tangent space $T^{1,0}X$ via $h^{T^{1,0}X}\left(v,w\right)=R^{\mathcal{L}}\left(v,\bar{w}\right)$,
$\forall v,w\in T^{1,0}X$. We also denote by $g^{TX}$ the underlying
Riemannian metric on $X$. The above metrics can be combined to define
an $L^{2}$-inner product on the $\Omega^{p,q}\left(X;\mathcal{L}\right)$.
As well as adjoint $\bar{\partial}_{\mathcal{L}}^{*}$ to the Dolbeault
operator with respect to these. The Kodaira Laplacian is the composition
\[
\boxempty_{\mathcal{L}}^{\left(p,q\right)}\coloneqq\bar{\partial}_{\mathcal{L}}^{*}\bar{\partial}_{\mathcal{L}}+\bar{\partial}_{\mathcal{L}}\bar{\partial}_{\mathcal{L}}^{*}
\]
which preserves the bi-degree $\left(p,q\right)$ of the form. The
Hodge theorem identifies the kernel of the Kodaira Laplacian with
the Dolbeault cohomology $\textrm{ker}\boxempty_{\mathcal{L}}^{\left(p,q\right)}=H^{p,q}\left(X;\mathcal{L}\right)$.
Of particular interest is the bidegree $\left(0,q\right)$. Here we
denote by the shorthand $\boxempty_{\mathcal{L}}^{q}\coloneqq\boxempty_{\mathcal{L}}^{\left(0,q\right)}$
and $H^{q}\left(X;\mathcal{L}\right)=H^{0,q}\left(X;\mathcal{L}\right)$.
The above quantities can be similarly defined for tensor powers $\mathcal{L}^{k}\coloneqq\mathcal{L}^{\otimes k}$
of the line bundle and its twists $F\otimes\mathcal{L}^{k}$ by another
auxiliary Hermitian, holomorphic line bundle $F$. 

For positive line bundles, the Kodaira and Serre vanishing theorems
state that one has
\begin{align}
H^{q}\left(X;K_{X}\otimes\mathcal{L}\right) & =0,\qquad\textrm{for }q>0,\label{eq:Kodaira vanishing}\\
H^{q}\left(X;F\otimes\mathcal{L}^{k}\right) & =0,\qquad\textrm{for }q>0\textrm{ and }k\gg0,\label{eq:Serre vanishing}
\end{align}
respectively. The above are proved using the Nakano estimates
\begin{align*}
\left\langle \boxempty_{K_{X}}^{q}s,s\right\rangle  & \geq q\left\Vert s\right\Vert ^{2},\quad\forall s\in\Omega^{0,q}\left(X;K_{X}\otimes\mathcal{L}\right),\\
\left\langle \boxempty_{F\otimes\mathcal{L}^{k}}^{q}s,s\right\rangle  & \geq\left(c_{1}k-c_{2}\right)\left\Vert s\right\Vert ^{2},\quad\forall s\in\Omega^{0,q}\left(X;\mathcal{L}^{k}\right),q>0,
\end{align*}
with the latter being claimed for some positive constants $c_{1},c_{2}>0$
that are independent of the power $k$.

\subsection{\label{subsec:Dirac-operators-and} Dirac operators and the eta invariant }

\noindent We now state some requisites about Dirac operators used
in the paper, a standard reference is \cite{Berline-Getzler-Vergne}.
Let $\left(X,g^{TX}\right)$ be a compact, oriented, Riemannian manifold
of odd dimension $n$. A spin structure is a $\textrm{Spin}\left(n\right)$
principal bundle $\textrm{Spin}\left(TX\right)\rightarrow SO\left(TX\right)$
that is an equivariant double covering of the $SO\left(n\right)$
principal bundle $SO\left(TX\right)$ of the orthonormal frames in
$TX$. There is a unique irreducible representation of $\textrm{Spin}\left(n\right)$
that gives rise to the associated spin bundle $S=\textrm{Spin}\left(TX\right)\times_{\textrm{Spin}\left(n\right)}S_{2m}$
. The Levi-Civita connection $\nabla^{TX}$ on the tangent bundle
$TX$ lifts to a connection on $\textrm{Spin}\left(TX\right)$. And
thus gives rise to the spin connection $\nabla^{S}$ on the spin bundle
$S$. The Clifford multiplication endomorphism $c:T^{*}X\rightarrow S\otimes S^{*}$
is the one arising from the standard representation of the Clifford
algebra of $T^{*}X$. It satisfies
\begin{align*}
c(a)^{2}=-|a|^{2}, & \quad\forall a\in T^{*}X.
\end{align*}
Next choose $\left(L,h^{L}\right)$ a Hermitian line bundle on $X$
along with a unitary connection $A_{0}$ on it. An additional one-form
$a\in\Omega^{1}(X;\mathbb{R})$ on $X$ gives rise to the family $\nabla^{h}=A_{0}+\frac{i}{h}a$,
$h\in\left(0,1\right]$ of unitary connections on $L$. We denote
the corresponding tensor product connection on $S\otimes L$ by $\nabla^{S\otimes L}\coloneqq\nabla^{S}\otimes1+1\otimes\nabla^{h}$.
This defines the coupled Dirac operator via
\begin{align*}
D_{h}\coloneqq hD_{A_{0}}+ic\left(a\right)=hc\circ\left(\nabla^{S\otimes L}\right):C^{\infty}(X;S\otimes L)\rightarrow C^{\infty}(X;S\otimes L)
\end{align*}
for $h\in\left(0,1\right]$. The Dirac operator $D_{h}$ is elliptic
and self-adjoint and thus has a discrete spectrum of eigenvalues. 

The eta function of $D_{h}$ is now defined by the formula below
\begin{align}
\eta\left(D_{h},s\right)\coloneqq & \sum_{\begin{subarray}{l}
\quad\:\lambda\neq0\\
\lambda\in\textrm{Spec}\left(D_{h}\right)
\end{subarray}}\textrm{sign}(\lambda)|\lambda|^{-s}=\frac{1}{\Gamma\left(\frac{s+1}{2}\right)}\int_{0}^{\infty}t^{\frac{s-1}{2}}\textrm{tr}\left(D_{h}e^{-tD_{h}^{2}}\right)dt,\label{eq:eta invariant definition}
\end{align}
$\forall s\in\mathbb{C}$. Here we have used the convention that $\textrm{Spec}(D_{h})$
is a multiset with each eigenvalue of $D_{h}$ being counted with
its multiplicity. The above series converges for $\textrm{Re}(s)>n.$
In \cite{APSI,APSIII} it is shown that the eta function \prettyref{eq:eta invariant definition}
has a meromorphic continuation to the entire complex $s$-plane and
further has no pole at zero. The eta invariant of the Dirac operator
$D_{h}$ is then defined to be its value at zero
\begin{equation}
\eta_{h}\coloneqq\eta\left(D_{h},0\right).\label{eq:eta invariant}
\end{equation}
We observe from \prettyref{eq:eta invariant definition} that the
above is formally the signature of the Dirac operator, i.e. the difference
between the number of its positive and negative eigenvalues. A variant
of the above, known as the reduced eta invariant, is defined by including
the zero eigenvalue 
\begin{align*}
\bar{\eta}_{h}\coloneqq & \frac{1}{2}\left\{ k_{h}+\eta_{h}\right\} \\
k_{h}\coloneqq & \textrm{dim ker }\left(D_{h}\right).
\end{align*}

\section{\label{sec:Spin-vanishing-and} Spin vanishing and Nakano estimate}

In this section we shall prove a variant of the Kodaira vanishing
theorem and Nakano estimate on a Fano manifold. These shall be used
in the next section in our computation of the eta invariant. 

To state the result, consider again $X$ our compact, complex manifold
of complex dimension $n$. As well as $\left(\mathcal{L},h^{\mathcal{L}}\right)\rightarrow X$
a positive holomorphic, Hermitian line bundle over it. Denote by $\nabla^{\mathcal{L}}$
the associated Chern conection on $\mathcal{L}$ and $R^{\mathcal{L}}\in\Omega^{1,1}\left(X\right)$
its curvature. The positivity of the bundle is defined by the condition
that $R^{\mathcal{L}}\left(w,\bar{w}\right)>0$ for each $w\in T^{1,0}X\setminus\left\{ 0\right\} $.
Then $\omega=\frac{i}{2\pi}R^{\mathcal{L}}\in\Omega^{1,1}\left(X\right)$
defines an integral Kahler form on the manifold whose cohomology class
$\left[\frac{i}{2\pi}R^{\mathcal{L}}\right]=c_{1}\left(\mathcal{L}\right)$
represents the first Chern class of the line bundle. 

The positivity of the bundle defines via $h^{T^{1,0}X}\left(v,w\right)=R^{\mathcal{L}}\left(v,\bar{w}\right)$,
for $v,w\in T^{1,0}X$, a Kahler metric on the complex tangent space.
We denote by $g^{TX}$ the corresponding associated Riemannian metric
on $X$.

A spin structure on $X$ corresponds to a holomorphic, Hermitian square
root $\mathcal{K}$ of the canonical line bundle $\mathcal{K}^{\otimes2}=K_{X}$.
The corresponding bundles of positive and negative spinors are $\Lambda^{\text{even}}T^{0,1*}\otimes\mathcal{K}$
and $\Lambda^{\text{odd}}T^{0,1*}\otimes\mathcal{K}$. We denote by
$\nabla^{\mathcal{K}}$ and $\nabla^{0,*}$ the corresponding Chern
connections on $\mathcal{K}$ and $\Lambda^{*}T^{0,1*}\otimes\mathcal{K}$
respectively. The Clifford multiplication map is given by 
\begin{align}
c & :T^{*}X\rightarrow\textrm{End}\left(\Lambda^{*}T^{0,1*}\otimes\mathcal{K}\right)\nonumber \\
c\left(v\right) & =\sqrt{2}\left(v^{1,0}\wedge-i_{v^{0,1}}\right),\quad\forall v\in T^{*}X,\label{eq:Clifford multiplication map}
\end{align}
while the spin Dirac operator is $D_{\mathcal{K}}=\sqrt{2}(\bar{\partial}_{\mathcal{K}}+\bar{\partial}_{\mathcal{K}}^{*})$.
Here $\bar{\partial}_{\mathcal{K}}$ is the holomorphic derivative
on $\Lambda^{*}T^{0,1*}\otimes\mathcal{K}$ while $\bar{\partial}_{\mathcal{K}}^{*}$
is the . A twisted Dirac operator is similarly defined 
\begin{equation}
D_{\mathcal{K}\otimes\mathcal{L}^{\otimes k}}=\sqrt{2}(\bar{\partial}_{\mathcal{K}\otimes\mathcal{L}^{k}}+\bar{\partial}_{\mathcal{K}\otimes\mathcal{L}^{k}}^{*}):\Omega^{0,*}\left(\mathcal{K}\otimes\mathcal{L}^{k}\right)\rightarrow\Omega^{0,*}\left(\mathcal{K}\otimes\mathcal{L}^{k}\right)\label{eq:Dirac operator base}
\end{equation}
acting on anti-holomorphic forms with valueed in the sections of $\mathcal{K}\otimes\mathcal{L}^{k}$
for each $k\in\mathbb{Z}$. The square of the above is the Kodaira
Laplacian $\boxempty_{\mathcal{K}\otimes\mathcal{L}^{k}}\coloneqq D_{\mathcal{K}\otimes\mathcal{L}^{k}}$.
This preserves the degree of the anti-holomorphic form and we denote
by $\boxempty_{\mathcal{K}\otimes\mathcal{L}^{k}}^{q}$ its restriction
to degree $q$.

Recall that the classical Kodaira vanishing theorem states 
\begin{equation}
H^{q}\left(X,K_{X}\otimes\mathcal{L}^{k}\right)=0,\qquad\textrm{for }q,k>0.\label{eq:Kodaira vanishing-1}
\end{equation}
Furthermore, it is obtained via Hodge theory and the Nakano estimate
\begin{equation}
\left\langle \boxempty_{K_{X}\otimes\mathcal{L}^{k}}^{q}s,s\right\rangle \geq qk\left\Vert s\right\Vert ^{2}\label{eq:Nakano estimate}
\end{equation}
for the Kodaira Laplacian $\boxempty_{K_{X}\otimes\mathcal{L}^{k}}^{q}:\Omega^{0,q}\left(K_{X}\otimes\mathcal{L}^{k}\right)\rightarrow\Omega^{0,q}\left(K_{X}\otimes\mathcal{L}^{k}\right)$
acting on anti-holomorphic $q$ forms. 

We shall now prove a variant of the above where the canonical bundle
is replaced by its square root. 
\begin{thm}
\label{thm: Spin-vanishing-and Nakano} (Spin vanishing and Nakano
inequality) Let $\mathcal{L}\rightarrow X$ be a positive line bundle
over a compact complex manifold and $\mathcal{K}$ a holomorphic,
Hermitian square root of the canonical bundle $K_{X}$. Suppose that
the Ricci curvature of the Kahler form $\omega=iR^{\mathcal{L}}$
satisfies $\textrm{Ric}_{\omega}\geq\kappa\omega$, for $\kappa\geq0$. 

Then one has the Nakano inequality 
\begin{align}
\left\langle \boxempty_{\mathcal{K}\otimes\mathcal{L}^{k}}^{q}s,s\right\rangle  & \geq q\left(k+\frac{1}{2}\kappa\right)\left\Vert s\right\Vert ^{2},\label{eq:Nakano inequality-1}\\
\left\langle \boxempty_{\mathcal{K}\otimes\mathcal{L}^{k}}^{q}s,s\right\rangle  & \geq\left(n-q\right)\left(-k+\frac{1}{2}\kappa\right)\left\Vert s\right\Vert ^{2},\label{eq:Nakano inequality 2}
\end{align}
$\forall s\in\Omega^{0,q}\left(\mathcal{K}\otimes\mathcal{L}\right)$,
for the Kodaira Laplacian $\boxempty_{\mathcal{K}\otimes\mathcal{L}^{k}}^{q}:\Omega^{0,q}\left(\mathcal{K}\otimes\mathcal{L}^{k}\right)\rightarrow\Omega^{0,q}\left(\mathcal{K}\otimes\mathcal{L}^{k}\right)$
acting on anti-holomorphic $q$-forms valued in the the square root
bundle $\mathcal{K}$.

Furthermore, one has the vanishing theorem 
\begin{align}
H^{q}\left(X;\mathcal{K}\otimes\mathcal{L}^{k}\right) & =\begin{cases}
0, & \textrm{for }q>0,\,k>-\frac{1}{2}\kappa,\\
0, & \textrm{for }q<n,\,k<\frac{1}{2}\kappa.
\end{cases}\label{eq:spin vanishing thm.}
\end{align}
\end{thm}

\begin{proof}
The proof uses the standard Bochner argument. Namely, the Bochner-Kodaira-Nakano
formula gives 
\begin{equation}
\boxempty_{\mathcal{K}\otimes\mathcal{L}^{k}}=\underbrace{\left(\nabla^{0,*}\right)^{*}\nabla^{0,*}}_{=\Delta^{0,*}}+c\left(R^{\mathcal{K}\otimes\mathcal{L}^{k}\otimes K_{X}^{*}}\right)\label{eq:Bochner Kodaira Nakano}
\end{equation}
\cite[1.4.63]{Ma-Marinescu}. Here the first term $\Delta^{0,*}$
is the Bochner Laplacian corresponding to the Chern connection on
$\Lambda^{*}T^{0,1*}\otimes\mathcal{K}$. While the second term $c\left(R^{\mathcal{K}\otimes\mathcal{L}^{k}\otimes K_{X}^{*}}\right)$
is the Clifford multiplication by the Chern curvature of the given
bundle, that may be written via 
\begin{equation}
c\left(R^{\mathcal{K}\otimes\mathcal{L}^{k}\otimes K_{X}^{*}}\right)=R^{\mathcal{K}\otimes\mathcal{L}^{k}\otimes K_{X}^{*}}\left(w_{j},\bar{w}_{k}\right)\bar{w}_{k}\wedge i_{w_{j}}\label{eq:Clifforn mult. in basis}
\end{equation}
in terms of an orthonormal basis $\left\{ w_{j}\right\} _{j=1}^{n}$
of the complex tangent space $T^{1,0}X$.

We now note that the Chern curvature of the anticanonical bundle $K_{X}^{*}$
and the square root $\mathcal{K}$ bundle are given 
\begin{equation}
R^{K_{X}^{*}}=-2R^{\mathcal{K}}=\textrm{Ric}_{\omega}\in\Omega^{1,1}\left(X\right)\label{eq:Chern curvature and Ricci}
\end{equation}
in terms of the Ricci curvature form of the metric $\omega=iR^{\mathcal{L}}$.
Using the lower bound $\textrm{Ric}_{\omega}\geq\kappa\omega$ on
the Ricci curvature, the above relations \prettyref{eq:Chern curvature and Ricci}
then give the inequality
\begin{equation}
\left\langle c\left(R^{\mathcal{K}\otimes\mathcal{L}^{k}\otimes K_{X}^{*}}\right)s,s\right\rangle \geq q\left(k+\frac{1}{2}\kappa\right)\left\Vert s\right\Vert ^{2}.\label{eq:Cliff mult ineq}
\end{equation}
Since the Bochner Laplacian $\Delta^{0,*}$ is a positive operator,
the Nakano inequality \prettyref{eq:Nakano inequality-1} now follows
from the above \prettyref{eq:Cliff mult ineq} via the Bochner-Kodaira-Nakano
formula \prettyref{eq:Bochner Kodaira Nakano}. The second Nakano
inequality \prettyref{eq:Nakano inequality 2} follows by duality
using the identity $*\boxempty_{\mathcal{K}\otimes\mathcal{L}^{k}}^{q}*=\boxempty_{\mathcal{K}\otimes\mathcal{L}^{k}}^{n-q}$.

The second part \prettyref{eq:spin vanishing thm.} relating to the
vanishing theorem now follows easily from the Nakano inequality \prettyref{eq:Nakano inequality-1}
and the Hodge theorem 
\[
H^{q}\left(X;\mathcal{K}\otimes\mathcal{L}^{k}\right)=\textrm{ker}\left(\boxempty_{\mathcal{K}\otimes\mathcal{L}^{k}}^{q}\right)
\]
on the Kahler manifold $X$.
\end{proof}

\section{\label{sec:Computation-of-the} Computation of the eta invariant}

We now prove our main theorem \prettyref{thm:main computation thm}
on the computation of the eta invariant. To this end one first needs
to decompose the Dirac operator along Fourier modes on the unit circle
bundle as in \cite[Sec. 5]{Savale-Asmptotics}. 

Thus let $Y=S^{1}\mathcal{L}$ be the unit circle bundle of the positive
line bundle $\mathcal{L}$. Denote by $S^{1}\rightarrow Y\xrightarrow{\pi}X$
the fibration and $\pi$ the natural projection onto the base. Next,
$TS^{1}\subset TY$ denotes the subbundle of vertical tangent vectors
and $T^{H}Y\subset TY$ the subbundle of horizontal tangent vectors
corresponding to the Chern connection $\nabla^{\mathcal{L}}$. These
give a circle invariant splitting 
\begin{equation}
TY=TS^{1}\oplus T^{H}Y.\label{eq:connection splitting}
\end{equation}
 A metric $g^{TS^{1}}$ on the vertical tangent space is defined by
setting the generator $e\in TS^{1}$ of the natural circle action
to have unit norm $\left\Vert e\right\Vert _{TS^{1}}=1$. While a
metric $g^{T^{H}Y}=\pi^{*}g^{TX}$ on the horizontal tangent space
is obtained by pullback of the Riemannian metric $g^{TX}$ on the
base. This defines the family of adiabatic metrics 
\begin{equation}
g_{\varepsilon}^{TY}=g^{TS^{1}}\oplus\varepsilon^{-1}\pi^{*}g^{TX},\quad\forall\varepsilon>0,\label{eq:def. adiabatic metrics}
\end{equation}
on $Y$ as in \cite{Bismut-Cheeger}. 

Let $\nabla^{TY,\varepsilon},\nabla^{TX}$ denote the Levi-Civita
connections of $g_{\varepsilon}^{TY},g^{TX}$ respectively. Let $p^{TS^{1}},p^{T^{H}Y}$
denote the projections of $TY$ onto $TS^{1},T^{H}Y$ summands respectively.
Define a connection on vertical bundle $TS^{1}$ via $\nabla^{TS^{1}}=p^{TS^{1}}\nabla^{TY,\varepsilon}$.
It was shown in \cite[Sec. 4]{Bismut-Cheeger} that the connection
$\nabla^{TS^{1}}$ is independent of $\varepsilon$. In the case of
circle bundles it can be computed by showing that the unit section
$e$ is $\nabla^{TS^{1}}$ -parallel $\nabla^{TS^{1}}e=0$ \cite[eq. 5.3]{Savale-Asmptotics}.
A second connection $\nabla$ on $TY$ is defined via $\nabla=\nabla^{TS^{1}}\oplus\pi^{*}\nabla^{TX}$.
The difference tensor of this with the Levi-Civita connection is set
to be 
\begin{equation}
S^{\varepsilon}\coloneqq\nabla^{TY,\varepsilon}-\nabla.\label{eq:diff. tensor}
\end{equation}
With $\left\langle ,\right\rangle _{\varepsilon}=g_{\varepsilon}^{TY}$
denoting the adiabatic metric, the above is computed in \cite[eq. 5.4]{Savale-Asmptotics}
to be 
\begin{equation}
\left\langle S^{\varepsilon}(U)V,W\right\rangle _{\varepsilon}=\frac{1}{2}\left[\left\langle T(V,W),U\right\rangle _{\varepsilon}-\left\langle T(W,U),V\right\rangle _{\varepsilon}-\left\langle T(U,V),W\right\rangle _{\varepsilon}\right],\label{eq:Difference tensor in terms of Torsion tensor}
\end{equation}
where $T$ is the torsion tensor of the connection $\nabla$. The
torsion tensor is further computed in \cite[Sec. 5]{Savale-Asmptotics}
to be given by 
\begin{align}
i_{e}T & =0\nonumber \\
T\left(\tilde{U}_{1},\tilde{U}_{2}\right) & =R^{\mathcal{L}}\left(U_{1},U_{2}\right)\label{eq:calc torsion tensor}
\end{align}
for $\tilde{U}_{1},\tilde{U}_{2}$ the horizontal lifts of two vector
fields $U_{1},U_{2}$ on the base $X$. A particular consequence of
the above calculation is the existence of the adiabatic limit of the
Levi-Civita connection $\nabla^{TY,0}\coloneqq\lim_{\varepsilon\rightarrow0}\nabla^{TY,\varepsilon}$
\cite{Bismut-Cheeger}.

A spin structure on $X$ corresponds to a holomorphic, Hermitian square
root $\mathcal{K}$ of the canonical line bundle $\mathcal{K}^{\otimes2}=K_{X}$.
The corresponding bundles of positive and negative spinors are $S_{+}^{TX}=\Lambda^{\text{even}}T^{0,1*}\otimes\mathcal{K}$
and $S_{-}^{TX}=\Lambda^{\text{odd}}T^{0,1*}\otimes\mathcal{K}$.
The spin connection is given by $\nabla^{0,*}$, the corresponding
Chern connection on $\Lambda^{*}T^{0,1*}\otimes\mathcal{K}$. The
spin structure lifts to $Y$ with the spin bundle $S^{TY}=\pi^{*}\left(S_{+}^{TX}\oplus S_{-}^{TX}\right)$
and connection being the lifts from the base. The space of spinors
then decomposes 
\begin{align}
C^{\infty}(Y,S^{TY}) & =\bigoplus_{k\in\mathbb{Z}}C^{\infty}(X;(S_{+}^{TX}\oplus S_{-}^{TX})\otimes\mathcal{L}^{\otimes k})\label{eq:Decomposition of spinors}\\
 & =\bigoplus_{k\in\mathbb{Z}}C^{\infty}(X;\Lambda^{*}T^{0,1*}\otimes\mathcal{K}\otimes\mathcal{L}^{\otimes k}),
\end{align}
with the $k$-th summand above corresponding to the eigenspace of
the generator $e$ on the unit circle \cite[eq. 5.9]{Savale-Asmptotics}.
With respect to the above decomposition \prettyref{eq:Decomposition of spinors},
and using the calculations \prettyref{eq:Difference tensor in terms of Torsion tensor}
and \prettyref{eq:calc torsion tensor}, the spin-c Dirac operator
\prettyref{eq:Spin Dirac operator} further has a decomposition given
by 
\begin{equation}
\text{\ensuremath{D_{r,\varepsilon}}}=\bigoplus_{k}\begin{bmatrix}k-\varepsilon(N-\frac{m}{2})-r & (2\varepsilon)^{1/2}\\
(2\varepsilon)^{1/2}(\bar{\partial}_{k}+\bar{\partial}_{k}^{*}) & -k+\varepsilon(N-\frac{m}{2})+r
\end{bmatrix}\label{eq:Dirac op decomposition}
\end{equation}
\cite[eq. 5.11]{Savale-Asmptotics}. Here we use the shorthand $\sqrt{2}(\bar{\partial}_{k}+\bar{\partial}_{k}^{*})=\sqrt{2}(\bar{\partial}_{\mathcal{K}\otimes\mathcal{L}^{k}}+\bar{\partial}_{\mathcal{K}\otimes\mathcal{L}^{k}}^{*})$
for the Dirac operator on the base \prettyref{eq:Dirac operator base},
while $N$ is the number operator which acts as multiplication by
$p$ on $\Lambda^{p}T^{0,1*}$.

The decomposition of the Dirac operator \prettyref{eq:Dirac op decomposition}
allows for the followind description of its spectrum in terms of the
base manifold. \footnote{Here we have corrected a sign in \prettyref{eq:Dirac op decomposition}
as well as the quadratic formula calculation of the Type eigenvalue
2 in \prettyref{prop:Computation of Dirac Spectrum} from \cite{Savale-Asmptotics}.} 
\begin{prop}
(\cite[Prop. 5.2]{Savale-Asmptotics}) \label{prop:Computation of Dirac Spectrum}The
eigenvalues of the spin-c Dirac operator \textup{$D_{r,\varepsilon}$
}\prettyref{eq:Spin Dirac operator} are given by the two types 

\begin{enumerate}
\item Type 1:
\begin{equation}
\lambda=(-1)^{q}(k-\varepsilon(q-\frac{n}{2})-r),\;0\leq q\leq n,k\in\mathbb{Z}\label{eq:Type 1}
\end{equation}
with multiplicity \textup{$h^{q,k}=dim\,H^{q}(X,\mathcal{K}\otimes\mathcal{L}^{\otimes k})$.}
\item Type 2:
\begin{equation}
\lambda=\frac{(-1)^{q+1}\varepsilon\pm\sqrt{(2k-\varepsilon(2q+1-n)-2r)^{2}+4\mu^{2}\varepsilon}}{2},\;0\leq q\leq n,k\in\mathbb{Z}\label{eq:Type 2}
\end{equation}
and $\frac{1}{2}\mu^{2}$ is a positive eigenvalue of $\boxempty_{\mathcal{K}\otimes\mathcal{L}^{k}}^{q}$.
The multiplicity of $\lambda$ is $d_{\mu}^{q,k}=e_{\mu}^{q,k}-e_{\mu}^{q-1,k}+\ldots+(-1)^{q}e_{\mu}^{0,k}$
where $e_{\mu}^{q,k}$ is the multiplicity of $\frac{1}{2}\mu^{2}$.
\end{enumerate}
\end{prop}

We shall now use the above in our calculation of the eta invariant.
To this end, let $\left\{ \nabla^{\delta}\right\} _{0\leq\delta\leq\varepsilon}$
be any family of connections on $TY$ such that $\nabla^{0}=\nabla^{TY,0},\nabla^{\varepsilon}=\nabla^{TY,\varepsilon}$.
This family determines a connection $\nabla^{TZ}$ on the tangent
bundle $TZ$ of $Z=Y\times[0,\varepsilon]_{\delta}$ via
\[
\nabla^{TZ}=d\delta\wedge\frac{\partial}{\partial\delta}+\nabla^{\delta}.
\]
 Let $R^{TZ}$ be the curvature of $\nabla^{TZ}$. By the Atiyah-Patodi-Singer
index theorem we have 
\begin{equation}
\eta^{r,\varepsilon}=\lim_{\varepsilon\rightarrow0}\eta^{r,\varepsilon}+2\left\{ \textrm{sf}\left\{ D_{r,\delta}\right\} _{0\leq\delta\leq\varepsilon}+\frac{1}{\left(2\pi i\right)^{m+1}}\int_{Z}\,\hat{A}(R^{TZ})\exp\left\{ rc\right\} \right\} .\label{eq:Transgression eta invariants}
\end{equation}

In the above, the first term is the adiabatic limit of the eta invariant
\cite{Bismut-Cheeger,Dai}. It was computed in \cite[Sec. 5.3.2]{Savale-Asmptotics}
to be 
\begin{align}
\lim_{\varepsilon\rightarrow0}\eta^{r,\varepsilon} & =\frac{1}{2}\int_{X}\hat{A}(X)\,\hat{\eta}_{r}\,\exp\left\{ rc\right\} ,\quad r\in\mathbb{Z},\label{eq:adiabatic eta}\\
\textrm{where }\qquad\hat{\eta}_{r} & =\begin{cases}
\frac{\exp\left((1-2\{r\})\frac{c}{2}\right)}{\sinh\left(\frac{c}{2}\right)}-\frac{1}{c/2}, & r\notin\mathbb{Z},\\
\left[\frac{\frac{c}{2}-\tanh\left(\frac{c}{2}\right)}{\frac{c}{2}\tanh\left(\frac{c}{2}\right)}\right], & r\in\mathbb{Z},
\end{cases}
\end{align}
via a modification of the arguments due to Zhang \cite{Zhang}. 

The last term is an integral of an transgression form involving the
$\hat{A}$-genus. On \cite[pg. 881]{Savale-Asmptotics}, the $\hat{A}$-genus
was computed to be 
\begin{align*}
\hat{A}(R^{TZ}) & =\Omega_{2}\exp\left\{ \Omega_{0}\right\} \\
\textrm{for }\quad\Omega_{0} & =2\textrm{tr}\left[p\left(R^{TX^{1,0}}+2i\delta\omega\right)\right]+2p\left(2i\delta\omega\right),\\
\textrm{and }\qquad\Omega_{2} & =2\textrm{tr}\left[ip'\left(R^{TX^{1,0}}+i2\delta\omega\right)\right]+i2p'\left(2i\delta\omega\right).
\end{align*}
 Thus the last integral term vanishes is computed to be 
\begin{equation}
\frac{1}{\left(2\pi i\right)^{m+1}}\int_{Z}\,\hat{A}(R^{TZ})\exp\left\{ rc\right\} =\int_{0}^{\varepsilon}d\delta\int_{X}\Omega_{2}\exp\left\{ \Omega_{0}\right\} \exp\left\{ rc\right\} .\label{eq:a hat vanishing}
\end{equation}

It now remains to compute the second term in the middle. This is the
spectral flow of the family of Dirac operators $\left\{ D_{r,\delta}\right\} _{0\leq\delta\leq\varepsilon}$
(see \cite[Ch. 3]{Savale-thesis2012}). Namely the number of eigenvalues
of $D_{r,\delta}$ that change sign from positive to negative as $\delta$
varies between $0$ and $\varepsilon$. The following theorem proves
that this spectral flow term vanishes under our hypotheses.
\begin{thm}
\label{thm:sf vanishes} Let $Y$ be the unit circle bundle of a positive
line bundle $\mathcal{L}\rightarrow X$ over a complex manifold of
real dimension $4m$. Suppose that the Ricci curvature of the Kahler
form $\omega=iR^{\mathcal{L}}$ satisfies $\textrm{Ric}_{\omega}\geq\kappa\omega$,
for $\kappa\geq0$. 

The the spectral flow for the family of spin-c Dirac operators vanishes
\begin{equation}
\textrm{sf}\left\{ D_{r,\delta}\right\} _{0\leq\delta\leq\varepsilon}=0\label{eq:sf vanishes}
\end{equation}
 for $\left|r\right|\leq\frac{1}{2}\kappa$.
\end{thm}

\begin{proof}
The proof is a consequence of the description of the spectrum from
\prettyref{prop:Computation of Dirac Spectrum} along with the Nakano
estimate \prettyref{thm: Spin-vanishing-and Nakano}.

Namely, first note by virtue of the spin vanishing theorem \prettyref{eq:spin vanishing thm.}
that the eigenvalues of type 1 \prettyref{eq:Type 1} do not change
sign for $\left|r\right|\leq\frac{1}{2}\kappa$. And hence these do
not contribute to the spectral flow.

As for the type 2 eigenvalues \prettyref{eq:Type 2}, note first by
virtue of the Nakano estimate \prettyref{eq:Nakano inequality-1}
that 
\begin{equation}
\frac{1}{2}\mu^{2}\geq\begin{cases}
q\left(k+\frac{1}{2}\kappa\right)\\
\left(n-q\right)\left(-k+\frac{1}{2}\kappa\right)
\end{cases}\label{eq:nakano ev est}
\end{equation}
 for $0\leq q\leq n$, $k\in\mathbb{Z}$ and for $\frac{1}{2}\mu^{2}$
a positive eigenvalue of $\boxempty_{\mathcal{K}\otimes\mathcal{L}^{k}}^{q}$.
Following this, elementary considerations using the formula \prettyref{eq:Type 2}
show that this type of eigenvalue does not change value either for
$\left|r\right|\leq\frac{1}{2}\kappa$. To spell this out, it suffices
to show 
\begin{align*}
 & -\varepsilon+\sqrt{(2k-\varepsilon(2q+1-n)-2r)^{2}+4\mu^{2}\varepsilon}\geq0\\
\iff & (2k-\varepsilon(2q+1-n)-2r)^{2}+4\mu^{2}\varepsilon\geq\varepsilon^{2}\\
\iff & \left[(2q+1-n)^{2}-1\right]\varepsilon^{2}+\left(2k-2r\right)^{2}\\
 & \qquad+\left[-2\left(2q+1-n\right)\left(2k-2r\right)+4\mu^{2}\right]\varepsilon\geq0,
\end{align*}
by squaring and expansion. When the complex dimension is even, the
$\varepsilon^{2}$ coefficient above is non-negative. It thus suffices
to show the $\varepsilon$ coefficient is non-negtive or that $\frac{1}{2}\mu^{2}\geq\left(q+\frac{1-n}{2}\right)\left(k-r\right)$
. This follows easily from \prettyref{eq:nakano ev est} obtained
by the Nakano estimates \prettyref{eq:Nakano inequality-1}, \prettyref{eq:Nakano inequality 2}. 

\end{proof}
Our main theorem \prettyref{thm:main computation thm} now follows
from \prettyref{eq:Transgression eta invariants} using \prettyref{eq:adiabatic eta},
\prettyref{eq:a hat vanishing} and \prettyref{eq:sf vanishes}.

\section{\label{sec:Manifolds-of-general-type} Manifolds of general type }

It is natural to ask whether the Fano hypothesis is necessary to obtain
our main result. Here we show that this is indeed the case. As the
spin vanishing theorem \prettyref{thm: Spin-vanishing-and Nakano},
the vanishing of the spectral flow \prettyref{thm:sf vanishes}, and
consequently our main theorem \prettyref{thm:main computation thm},
do not hold over a manifold $X$ of general type; or when the anticanonical
bundle $K_{X}^{*}$ is negative.

To give the example, let 
\begin{equation}
X=\left\{ z|p\left(z\right)=0,\,\textrm{deg}p=d\right\} \subset\mathbb{CP}^{n+1}\label{eq:degree d hyp.}
\end{equation}
be a degree $d$ smooth hypersurface of projective space. Let $\mathcal{O}\left(1\right),\mathcal{O}\left(-1\right)$
be the hyperplane and tautological line bundles over projective space.
The anticanonical line bundle of $\mathbb{CP}^{n+1}$ is known to
be $K_{\mathbb{CP}^{n+1}}^{*}=\mathcal{O}\left(n+2\right)$. By the
adjunction formula, the anticanonical bundle of the hypersurface is
computed to be
\begin{align}
K_{X}^{*} & =K_{\mathbb{CP}^{n+1}}^{*}\otimes\mathcal{O}_{X}\left(-d\right)\nonumber \\
 & =\mathcal{O}_{X}\left(n+2-d\right).\label{eq:anti-canonical bundle hyp.}
\end{align}
Thus $K_{X}^{*}$ is negative for $d>n+2$. Moreover, for even degree
$d$, it has a square root given by $\mathcal{K}^{*}=\mathcal{O}_{X}\left(\frac{n}{2}+1-\frac{d}{2}\right)$.
Or given via $\mathcal{K}=\mathcal{O}_{X}\left(-\frac{n}{2}-1+\frac{d}{2}\right)$
as a square root of the canonical bundle that defines the spin structure. 

We now choose as $\mathcal{L}=\mathcal{O}_{X}\left(1\right)\rightarrow X$
the restriction of the hyperplane line bundle. This gives $\mathcal{K}\otimes\mathcal{L}^{k}=\mathbb{C}$
and hence 
\begin{equation}
H^{0}\left(X;\mathcal{K}\otimes\mathcal{L}^{k}\right)=H^{0}\left(X;\mathbb{C}\right)\neq0,\quad\textrm{for }k=\frac{n}{2}+1-\frac{d}{2}<0,\label{eq:nonvanishing in eg.}
\end{equation}
since it contains the constant functions. This shows that the spin
vanishing theorem \prettyref{eq:spin vanishing thm.} does not hold
in this example. As a consequence the eigenvalue $\lambda=(-1)^{q}(k-\varepsilon(q-\frac{n}{2})-r)$
of type 1 in \prettyref{eq:Type 1} crosses the origin at $\varepsilon=-\frac{2k}{n}$
for $q=r=0$. Thus the vanishing of the spectral flow \prettyref{thm:sf vanishes}
does not hold. And our formula for the eta invariant in \prettyref{thm:main computation thm}
would have to change.

\bibliographystyle{siam}
\bibliography{biblio}

\end{document}